\documentclass[reqno,12pt]{amsart}
\usepackage{amsmath,amsthm,enumerate,graphics,pstricks,amsfonts,latexsym,amsopn,verbatim,amscd,amssymb}
\usepackage{color}
\usepackage{psfrag}
\usepackage[all]{xy}

\theoremstyle{plain}
\newtheorem{thm}{Theorem}[section]

\theoremstyle{definition}

\begin{document} 

\title[On finite groups with nine centralizers]{On finite groups with nine centralizers} 

\author[S. J. Baishya  ]{Sekhar Jyoti Baishya*} 
\address{S. J. Baishya, Department of Mathematics, North-Eastern Hill University,
Permanent Campus, Shillong-793022, Meghalaya, India.}

\email{sekharnehu@yahoo.com}

\begin{abstract}
Given a finite group $G$, let $Cent(G)$ denote the set of distinct centralizers of elements of $G$. The group $G$ is called $n$-centralizer if $|Cent(G)|=n$ and primitive $n$-centralizer if $|Cent(G)|=|Cent(\frac{G}{Z(G)})|=n$. In this paper, we characterize the $9$-centralizer and the primitive $9$-centralizer groups.
\end{abstract}

\subjclass[2010]{11A25, 20D60, 20E99}
\keywords{Finite groups, $n$-centralizer groups, primitive $n$-centralizer groups}
\thanks{*I am deeply indebted to my supervisor Prof. Ashish Kumar Das for his constant encouragement throughout my M.Phil and Ph.D career.}
\maketitle

\section{Introduction} \label{S:intro}

In this paper, all groups are finite  and all notations are usual. For example $C_n$ denotes the cyclic group of order $n$, $Z(G)$ denotes the center of a group $G$, $D_{2n}$ denotes the dihedral group of order $2n$, $ C_n \rtimes C_p$ denotes the semidirect product of $C_n$ and $C_p$ and $(C_6, C_7)$ denotes the Frobenius group with complement $C_6$ and the kernel $C_7$. A finite group $G$ is said to be a $CA$-group if $C(x)$ is abelian for all $x \in G \setminus Z(G)$.

Given a finite group $G$, let $Cent(G)$ denote the set of centralizers of $G$, i.e., $Cent(G)=\lbrace C(x) \mid x \in G\rbrace $, where $C(x)$ is the centralizer of the element $x$ in $G$. The group $G$ is called $n$-centralizer if $|Cent(G)|=n$ and primitive $n$-centralizer if $|Cent(G)|=|Cent(\frac{G}{Z(G)})|=n$. The study of finite groups in terms of $|Cent(G)|$, becomes an interesting research topic in the recent years.  Starting with Belcastro and Sherman in 1994 \cite{ctc092}, many authors have studied the influence of $|Cent(G)|$ on a finite group $G$ (see [1], [3--7] and [13--15]). It is clear that a group is $1$-centralizer if and only if it is abelian. In \cite{ctc092}, Belcastro and Sherman proved that there is no $n$-centralizer group for $n=2, 3$. On the otherhand, A .R. Ashrafi in \cite {en09} proved that there exists $n$-centralizer groups for $n \neq 2, 3$. The finite $n$-centralizer groups for $n=4, \dots, 8$ has been characterized (see \cite{ctc092}, \cite{ctc09}, \cite{ed09}). In \cite{baishya}, we characterized finite odd order $9$-centralizer groups. 

In this paper we continue with this problem and prove that $G$ is a finite $9$-centralizer group if and only if $\frac{G}{Z(G)}\cong C_7 \rtimes C_2$ or $C_7 \rtimes C_3$ or $(C_6, C_7)$ or $C_7 \times C_7$. As a consequence we also characterize the primitive $9$-centralizer finite groups.

\section{The main results}

In this section we prove the main results of the paper:

\begin{thm}\label{thm1}
Let $G$ be a finite group. Then $G$ is a $9$-centralizer group if and only if $\frac{G}{Z(G)}\cong C_7 \rtimes C_2$ or $C_7 \rtimes C_3$ or $(C_6, C_7)$ or $C_7 \times C_7$.
\end{thm}

\begin{proof}
Let $G$ be a  finite $9$-centralizer group. Let $\lbrace x_1, x_2, \dots , x_r\rbrace $ be a set of pairwise non-commuting elements of $G$ having maximal size. 
Suppose $X_i= C(x_i)$, $1 \leq i \leq r$ and $|G: X_1| \leq |G: X_2| \leq \dots \leq |G: X_r|$. By \cite[Lemma 2.4]{ed09}, we have $5 \leq r \leq 8$. 

Now, suppose $r=5$. By \cite[Lemma 2.6]{ed09}, $G$ is not a $CA$-group. Therefore in view of  \cite[Remark 2.1]{ed09}, we have $|G: Z(G)| = 16$, otherwise $G$ will be a $CA$-group. It follows that 
\[
G=Z(G) \sqcup y_1Z(G)\sqcup \dots \sqcup y_{15}Z(G), 
\]
where  $y_i \in G \setminus Z(G)$ and $1 \leq i \leq15$.
It can be easily verify that $G$ has a centralizer of index $2$, otherwise $G$ will be a $CA$-group. Without any loss, we may assume that $|C(y_1)|=\frac{|G|}{2}$. Let
\[
C(y_1)=Z(G) \sqcup y_1Z(G)\sqcup \dots \sqcup y_7Z(G).
\]
Now, suppose $y \in G \setminus C(y_1)$ and $|C(y)|=\frac{|G|}{2}$. Without any loss, we may assume that $y=y_8$. Clearly, $C(y_1) \cap C(y_8) \neq Z(G)$, otherwise $|G: Z(G)|=4$, which is a contradiction. Next, suppose $|C(y_1) \cap C(y_8)|= 2|Z(G)|$. Then there exists some $y_iZ(G), 2 \leq i \leq 7$ such that $y_iZ(G) \subseteq C(y_1) \cap C(y_8)$. But then $|C(y_i)|=\frac{|G|}{2}$ and $C(y_i)$ will be different from $C(y_1)$ and $C(y_8)$, noting that $|Z(C(y_1))|=|Z(C(y_8))|=|Z(C(y_i))|=2|Z(G)|$. In this situation one can easily see that 
\[
G=C(y_1) \cup C(y_8)\cup C(y_i).
\]
Then by \cite[Theorem 1]{bryan},  we have $|\frac{G}{C(y_1) \cap C(y_8)\cap C(y_i)}|=4$, forcing $|\frac{C(y_1) \cap C(y_8)\cap C(y_i)}{Z(G)}|=4$, which is impossible. Finally,  suppose $|C(y_1) \cap C(y_8)|= 4|Z(G)|$. Then there exists some $y_{i_1}Z(G), y_{i_2}Z(G), y_{i_3}Z(G),  2 \leq i_1< i_2< i_3 \leq 7$ such that
\[
y_{i_1}Z(G), y_{i_2}Z(G), y_{i_3}Z(G) \subseteq C(y_1) \cap C(y_8).
\]
In the present situation also one can easily verify that 
\[|C(y_1)|=|C(y_8)|=|C(y_{i_1})|=|C(y_{i_2})|=|C(y_{i_3})|=\frac{|G|}{2},
\]
and all of the above centralizers are distinct, noting that the size of the centers of each of the above centralizers  is $2|Z(G)|$. Now, considering the centralizers of $y_i$'s, $i \in \lbrace 2, 3, \dots, 15 \rbrace \setminus \lbrace 1, i_1, i_2, i_3, 8 \rbrace$, one can verify that $|Cent(G)|>9$, which is a contradiction. Thus we have seen that $|C(y_8)|\leq \frac{|G|}{4}$. Therefore $y_8, \dots, y_{15}$ will give at least $4$ proper distinct centralizers of $G$ other that $C(y_1)$. Now, considering the centralizers of $y_2, \dots, y_7$, one can see that $|Cent(G)|> 9$, which is again a contradiction.

Next, suppose $r=6$. By \cite[Lemma 2.6]{ed09}, $G$ is not a $CA$-group. Again, by \cite[Remark 2.1]{ed09}, we have $G=X_1 \cup \dots \cup X_6$ and by \cite[Proposition 2.5]{ed09}, $X_i$'s are abelian for all $1 \leq i \leq 6$. Moreover, it folllows from \cite[Lemma 3.3]{ctc094} that $|G:X_2|\leq 5$ and from \cite[Remark 2.1]{ed09} that $|G: Z(G)|\leq 36$. Therefore the possible values of $|G: Z(G)|$ are $16, 24, 32$ and $36$, noting that if $|G: Z(G)|=pqr$, where $p, q, r$ are primes, then $G$ is a $CA$-group. It is easy to see from \cite[Proposition 2.5]{ed09} that $|G:X_1|=|G:X_2|=4$, otherwise $|G: Z(G)| \leq 12$.  

Now, suppose $x \in (X_1 \cap X_2) \setminus Z(G)$. Then 
\[
G= C(x) \cup X_3 \cup X_4 \cup X_5 \cup X_6.
\]
It follows from \cite[Lemma 3.3]{ctc094}, that  $|G:X_3| \leq 4$. Now, note that $X_i \cap X_j =Z(G)$ for any $3 \leq i, j \leq 6$, $i \neq j$. Otherwise by \cite[Lemma 3.3]{ctc094} we get $|G:X_k|\leq 3$, for some $3 \leq k \leq 6$. But then  $|G: Z(G)| \leq 6$ (by \cite[Proposition 2.5]{ed09}), which is impossible. Also, note that $C(x) \cap X_l \neq Z(G)$, for any $3 \leq l \leq 6$, otherwise $|G: Z(G)|\leq 8$, which is again a contradiction. Let $a \in (C(x) \cap X_l) \setminus Z(G)$, where $3 \leq l \leq 6$. Then $|C(a)|=|C(x)|=\frac{|G|}{2}$. Also, note that $C(a) \neq C(x)$, otherwise $X_l \subseteq C(x)$, and hence $|G: X_m| \leq 3$ for some $3 \leq m \leq 6$ (by \cite[Lemma 3.3]{ctc094}). But then $|G:Z(G)| \leq 6$ (by \cite[Proposition 2.5]{ed09}), which is a contradiction. It now easily follows that $|Cent(G)|\neq 9$.
Thus we have seen that $X_1 \cap X_2=Z(G)$ and hence $|G: Z(G)|= 16$. Now, using arguments similar to the case of $r=5$, we get a contradiction.

Finally, suppose $r=7$. In this case also, by \cite[Lemma 2.6]{ed09}, $G$ is not a $CA$-group. Again, by \cite[Remark 2.1]{ed09}, we have $G=X_1 \cup \dots \cup X_7$ and by \cite[Proposition 2.5]{ed09}, $X_i$'s are abelian for all $1 \leq i \leq 7$.  Now, suppose $K=\langle X_1, X_2, X_3 \rangle \lneq G$. Then 
\[
G=K \cup X_4 \cup X_5 \cup X_6 \cup X_7.
\]
It folllows from \cite[Lemma 3.3]{ctc094} that $|G:X_4|\leq 4$,  and hence by \cite[Proposition 2.5]{ed09}, we have $|G:X_4|=4$, otherwise $|G: Z(G)| \leq 9$ and $G$ will be a $CA$-group. Therefore in view of  \cite[Proposition 2.5]{ed09} again, it follows that $|G:X_1|= \dots =|G:X_7|=4$ and $G$ has a centralizer of index $2$, say $C(b)$ for some $b \in G$. Now, it is easy to see that $C(a) \cap X_i = Z(G)$ for some $1 \leq i \leq 7$, otherwise $X_i \subsetneq C(a)$ for all $1 \leq i \leq 7$, which is impossible. But then $|G : Z(G)|=8$, which is again impossible. Hence $\langle X_1, X_2, X_3 \rangle = G$ and by \cite[pp. 857]{ctc094}, we have  $|G: Z(G)| \leq 36$.

By \cite[Proposition 2.5]{ed09}, there exists a proper non-abelian centralizer, say $C(z)$ for some $z \in G$, which contains $X_{i_i}, X_{i_2}$ and $X_{i_3}$ for three distinct $i_1, i_2, i_3 \in \lbrace 1, \dots, 7 \rbrace$. Then 
\[
G=C(z) \cup X_{j_1} \cup X_{j_2} \cup X_{j_3} \cup X_{j_4},
\]
for four distinct $j_1, j_2, j_3, j_4 \in \lbrace 1, \dots, 7\rbrace$. In view of \cite[Lemma 3.3]{ctc094}, we have $|G:X_{j_1}| \leq 4$ and by \cite[Proposition 2.5]{ed09}, $|G:X_{j_1}| = |G:X_{j_2}|=|G:X_{j_3}|=|G:X_{j_4}|=4$.   Again, using \cite[Proposition 2.5]{ed09}, it is easy to see that $X_{j_k} \cap X_{j_l}=Z(G)$ for any $1 \leq k, l \leq 4$ with $k \neq l$. But then $C(z) \cap X_{j_k}=Z(G)$ for some $k \in \lbrace 1, \dots, 4 \rbrace$, and hence $|G: Z(G)| \leq 8$, which is again a contradiction. 

Therefore $r =8$. Now, by \cite[Lemma 2.6]{ed09}, $G$ is a $CA$-group and hence by \cite[Remark 2.1]{ed09} we have $X_i \cap X_j = Z(G)$ for  $1 \leq i, j \leq 8, i \neq j$. Moreover by \cite[Theorem A]{zarrin0942} we have $G$ is a solvable group and hence $\frac{G}{Z(G)}$ is either a $p$ group for some prime $p$ or a Frobenius group ( see \cite[Theorem 3.10]{ed009}). By  \cite[Proposition 2.5]{ed09}, We have
$G=X_1 \cup X_2 \cup \dots \cup X_8.$
Also, by \cite[Lemma 3.3]{ctc094}, we have $|G : X_2| \leq 7$.

Suppose $|G : X_2| = 6$. Then  $|G : Z(G)| \leq 36$. In view of some known results (see \cite{ed09}, \cite{en09}, \cite{baishya}), it follows that the possible values of $|G : Z(G)|$ are $36, 24$ and $18$.
 Now, suppose $|G: Z(G)| =36$. Clearly, $|G: X_1| = 6$, otherwise $|G: Z(G)| < 36$, which is not possible. Now, using \cite[Theorem 1]{ctc093}, it is easy to see that $|X_3| = |X_4|=|X_5|=\frac{|G|}{6}$. In this situation, if $|G: X_6|=6$, then again using \cite[Theorem 1]{ctc093} we get $|G: X_7|=9$ and $|G: X_8|=12$. One can easily see that $X_7 \lhd G$ and so $X_1X_7 \leq G$. But $|X_1X_7|=\frac{2|G|}{3}$, which is a absurd. Therefore by \cite[Theorem 1]{ctc093}, we have $|X_6| = |X_7|=|X_8|=\frac{|G|}{9}$. But then  $|G : Z(G)| \neq 36$, which is a contradiction.
 
 Next, suppose $|G: Z(G)|=24$. Then $\frac{G}{Z(G)} \cong S_4$ and hence $G$ has atleast $4$ centralizers of index $8$. Now, using \cite[Theorem 1]{ctc093}, one can verify that $|G: X_2|=|G : X_3|=|G: X_4|=6$ and $|G: X_5|=\dots =|G: X_8|=8$. Therefore, again using \cite[Theorem 1]{ctc093}, we get $G=X_1X_5$ and hence $|G: X_1|=3$. But then, $G=X_1X_2$ and so $|G: Z(G)|=18$, which is impossible.
 
 Finally, suppose $|G: Z(G)| =18$. Then $\frac{G}{Z(G)} \cong D_{18}$, and by \cite[Proposition 2.2]{ed09} we have $|Cent(G)|=11$, which is not possible. Thus, we have seen that $|G: X_2| \neq 6$.
 
 Now, suppose  $|G: X_2| = 5$, then  $|G: Z(G)| \leq 25$. In the present situation also, in view of some known results (see \cite{ctc092}, \cite{ed09}), we can see that $|Cent(G)| \neq 9$.

Next, suppose Suppose $|G: X_2| =4$. Then  $|G : Z(G)| = 16$, noting that $|G: Z(G)| > 15$ (see \cite{ctc092}, \cite{ed09}). Clearly, we must have $|X_1| = \frac{|G|}{4}$. Now, by calculating the number of cosets of $Z(G)$, in the $X_i$'s where $1 \leq i \leq 8$, one can easily get a contradiction. 

Finally, suppose  $|G : X_2| \leq 3$. Then  $|G: Z(G)| \leq 9$, and hence $|Cent(G)| \neq 9$ (by \cite[Theorem 5]{ctc092}).

Therefore $|G: X_2| = 7$ and so by \cite[Lemma 3.3]{ctc094},  $|G: X_2| = \dots =|G: X_8|= 7$. Moreover, by \cite[Theorem 1]{ctc093}, we have, $G=X_1X_2$. It follows that 
$|G: X_1| = 2, 3, 6$ or $7$. Consequently, 
$\frac{G}{Z(G)}\cong C_7 \rtimes C_2$ or $C_7 \rtimes C_3$ or $(C_6, C_7)$ or $C_7 \times C_7$. 

Conversely, if $\frac{G}{Z(G)}\cong (C_6, C_7)$, then using Correspondence theorem and  \cite[Problem 7.1]{isaacs}, one can see that $|Cent(G)|=9$. Again, if $\frac{G}{Z(G)}\cong C_7 \rtimes C_2$ or $C_7 \rtimes C_3$ or $C_7 \times C_7$, then by \cite[Corollary 2.5]{baishya}, we have $|Cent(G)|=9$.
\end{proof}

As a consequence we obtain the following result for primitive $9$-centralizer groups:

\begin{thm}\label{thm2}
Let $G$ be a finite group. Then $G$ is a primitive $9$-centralizer group if and only if $\frac{G}{Z(G)}\cong C_7 \rtimes C_2$ or $C_7 \rtimes C_3$ or $(C_6, C_7)$.
\end{thm}

\begin{proof}
Using \cite[Problem 7.1]{isaacs}, we can see that $|Cent((C_6, C_7))|=9$. Moreover, it can be easily verify that  $|Cent(C_7 \rtimes C_3)|=9$  and $|Cent(C_7 \rtimes C_3)|=9$. Now, the result follows from Theorem \ref{thm1}.
\end{proof}


\end{document}